\documentclass[11pt]{amsart}

\usepackage[T1]{fontenc}
\usepackage{lmodern}
\usepackage{microtype}
\usepackage{amsmath,amssymb,mathtools}
\usepackage{mathrsfs}
\usepackage{enumitem}
\usepackage{tabularx}
\usepackage[hidelinks]{hyperref}

\numberwithin{equation}{section}

\newtheorem{theorem}{Theorem}[section]
\newtheorem{prop}[theorem]{Proposition}
\newtheorem{lemma}[theorem]{Lemma}

\theoremstyle{remark}

\newcommand{\M}{\mathcal M}
\newcommand{\N}{\mathbb N}
\newcommand{\Nzero}{\mathbb N_0}
\newcommand{\cP}{\mathcal P}
\newcommand{\dd}{\,\mathrm d}
\newcommand{\hUC}{h_{\mathrm{top}}^{\mathrm{UC}}}
\newcommand{\hP}{h_{\mathrm{top}}^{P}}
\newcommand{\hB}{h_{\mathrm{top}}^{B}}
\newcommand{\htop}{h_{\mathrm{top}}}

\newcommand{\TV}{\mathrm{TV}}

\title[
]
{Entropies of compact subsets and supported measures 
}

\author{Qiang Huo}
\address{$^1$ School of Mathematical Sciences \& State Key Laboratory of Cognitive Intelligence, University of Science and Technology of China, Hefei, Anhui, 230026, P. R. China}
\email{qianghuo@ustc.edu.cn}

\author{Xiangtong Wang}
\address{$^2$ School of Mathematical Sciences \& State Key Laboratory of Cognitive Intelligence, University of Science and Technology of China, Hefei, Anhui, 230026, P. R. China.}
\email{wxt2020@mail.ustc.edu.cn}
\date{\today}
\begin{document}
\begin{abstract}
Let $(X,T)$ be a topological dynamical system and  $(\M(X),T_*)$ be its induced system.  For a non-empty compact subset $K\subset X$,  we define $\M(K)$ as the  set of Borel probability measures supported on $K$. 
 In this paper, we systematically study the relationship between various entropies of $(T,K)$ and of $(T_*,\M(K))$. We show that:
 
 \begin{equation*}
 	\begin{aligned}
 & \hUC(T,K)>0 \iff \hUC(T_*,\mathcal{M}(K))=+\infty, \\
 &\hP(T,K)>0\iff \hP(T_*,\mathcal{M}(K))>0,\\
 &\hB(T,K)>0 \implies \hB(T_*,\mathcal{M}(K))>0 ,
\end{aligned}
\end{equation*}
where $\hUC(T,K)$, $\hP(T,K)$, and $\hB(T,K)$ denote the upper capacity topological entropy, the packing topological entropy, and the Bowen topological entropy of $K$, respectively. Additionally,  we present a counterexample involving a non-invariant set, demonstrating that the converse of the third assertion is not valid in general.

\end{abstract}

\maketitle

\section{Introduction}
\label{sec:introduction}
Topological entropy, introduced by Adler, Konheim and McAndrew \cite{AKM65}, is a fundamental topological invariant which quantifies the average complexity of orbits over the entire phase space. Let $(X,T)$ be a topological dynamical system and $(\M(X),T_*)$ be its induced system on the space of Borel probability measures. It is natural to investigate the relation between  the complexity of  $(X,T)$ and that of $(\M(X),T_*)$. The pioneering work of Bauer and Sigmund \cite{BauerSigmund} showed that positive topological entropy of $(X,T)$ forces  $(\M(X),T_*)$ to have infinite topological entropy. This phenomenon demonstrates that a dramatic amplification of dynamical complexity may occur when passing to the induced system.  Glasner and Weiss \cite{GlasnerWeiss} further proved that $(X,T)$ has zero topological entropy if so does $(\M(X),T_*)$.
Their results together yield a dichotomy for the topological entropy of an induced system - it is either zero or infinite.

Along this line, a series of subsequent research have been carried out.
For instance, Liu, Qiao and Xu \cite{LQX20} studied topological entropies of nonautonomous dynamical systems and their induced systems, and obtained that a nonautonomous dynamical system has positive topology entropy if and only if its induced system has infinite topological entropy.
Liu and Qiao \cite{LQrelative} established a relative version of the above dichotomy by proving that a factor map  has positive relative topological entropy if and only if the induced factor map has infinite relative topological entropy. In addition to topological entropy, there are other topological invariants - such as topological sequence entropy and entropy dimension - that quantify the complexity of a system in different ways.  We refer the interested readers to \cite{KL05,QZ17}, where the sequence entropy and entropy dimension of induced systems are systematically investigated.

Since dynamical systems often exhibit dramatically different behaviors in distinct regions, global entropy fails to reflect local structure. This motivates the development towards the local entropy theory \cite{BR24, Blan93, BK83, DZZ23, HY06, LZ26, WWZ22}. In this paper, we focus specifically on Bowen entropy and packing entropy.
Bowen \cite{Bowen1973} introduced the notion of topological entropy for arbitrary subsets in a given dynamical system, which serves as the dynamical analogue of Hausdorff dimension. Later, Feng and Huang \cite{FengHuang} proposed  packing entropy as the dynamical counterpart of packing dimension, as well as upper capacity entropy defined via spanning sets and separated sets. In the same paper, they further introduced local measure-theoretic  entropy, and established variational principles linking local topological entropy (Bowen and packing) to local measure-theoretical entropy. Notably, for any subset, the packing entropy is bounded below by the Bowen entropy and above by the upper capacity entropy. Moreover, these three quantities coincide for any invariant subsets.

For a non-empty compact set $K\subset X$, denote by $\M(K)$ the set of Borel probability measures supported on $K$. As $K$ is not necessarily $T$-invariant, the pair $(T_*,\M(X))$ need not be a dynamical system. 
The significance of the Feng--Huang variational principles lies in the fact that, from the local perspective, the topological complexity of a subset is characterized by the complexity of those measures supported on it. This idea motivates us to investigate the complexity of an arbitrary compact susbet and of the Borel probability measures supported on it.
In particular, it is natural to ask whether one can relate the Bowen, packing and upper capacity entropies of $K$ to the corresponding entropies of $\M(K)$.

Our first result establishes a dichotomy for the upper capacity entropy of $\M(K)$. 
\begin{theorem}
	\label{thm:intro-upper-capacity}
	Let $(X,T)$ be a  topological dynamical system.
	If $K\subset X$  is non-empty and compact, the following statements hold:
	\begin{align*}
		& \hUC(T,K)=0 \iff \hUC(T_*,\mathcal{M}(K))=0, \\
		& \hUC(T,K)>0\iff \hUC(T_*,\mathcal{M}(K))=+\infty.
	\end{align*}

\end{theorem}

Our second result is concerned with the packing entropy of $K$ and $\M(K)$.

\begin{theorem}
	\label{thm:intro-packing}
	Let $(X,T)$ be a topological dynamical system. 
	If $K\subset X$  is non-empty and compact, the following equivalence holds:
	\begin{align*}
		\hP(T,K)>0\iff \hP(T_*,\mathcal{M}(K))>0.
	\end{align*}
	
\end{theorem}

The proof is strikingly different from that of Glasner and Weiss \cite{GlasnerWeiss} and highly depends on the measure-theoretical local entropy introduced by Feng and Huang \cite{FengHuang}. We also apply the Feng--Huang variational principle to $K$ and $\M(K)$ alternately.

The third result on Bowen entropy exhibits a unexpected phenomenon.

\begin{theorem}
	\label{thm:intro-bowen}
	Let $(X,T)$ be a  topological dynamical system.
	If $K\subset X$  is non-empty,  then  
	\[
	\hB(T,K)>0 \implies \hB(T_*,\mathcal{M}(K))=+\infty.
	\]
	Furthermore, there exist a topological dynamical system $(X',T')$ and a non-empty compact $K'\subset X'$ such that
	\[
	\hB(T',K')=0
	\  \text{whereas}\ 
	\hB(T'_*,\M(K'))=+\infty.
	\]
\end{theorem}

The paper is organized as follows. In Section~\ref{sec:definitions}, we recall the definitions and basic properties of upper capacity, packing and Bowen topological entropies for arbitrary subsets, together with the corresponding variational principles for Bowen and packing  topological entropies. In Section~\ref{sec: proof of thm1.1}, we prove Theorem~\ref{thm:intro-upper-capacity}. In Section~\ref{sec:proof of thm1.2}, we prove Theorem~\ref{thm:intro-packing}. In Section~\ref{sec: proof of thm1.3}, we prove Theorem~\ref{thm:intro-bowen}, including the construction of an example showing that the converse implication for Bowen entropy fails in general.

\section{Preliminaries}
\label{sec:definitions}
 For clariﬁcation, throughout this paper by a {\bf topological dynamical system} (TDS for short) we mean a pair $(X, T )$, where $X$ is a compact metric space endowed with a metric $d$ and $T\colon X\to X$ be continuous.  We denote by $\N$ and $\N_+$ the sets of nonnegative integers and positive integers, respectively. For a non-empty  subset $K\subset X$, define
 $$
 \M(K):=\{\mu\in\mathcal{M}(X):\operatorname{supp}\mu\subset K\},
 $$
 the space of Borel probability measures supported on $K$. 
 \subsection{Induced systems}
  Let $(X,d)$ be a compact metric space, and let $\mathcal{M}(X)$ be the space of Borel probability measures on $X$ endowed with the weak*-topology. It is classical that $\mathcal{M}(X)$ is  compact and metrizable \cite[Theorem 6.4]{Par67}. Fix a sequence $\{g_\ell\}_{\ell\ge1}$  dense in the closed unit ball of $C(X)$, and define a  compatible metric $D$ on $\mathcal{M}(X)$ by
  \begin{equation}\label{eq:weak* metric}
  D(\mu,\nu):=\sum_{\ell=1}^{\infty}2^{-\ell}\left|\int_X g_\ell\,d\mu-\int_X g_\ell\,d\nu\right|, \quad\forall \mu,\nu\in M(X).
  \end{equation}
  
The  push-forward map $T_*:\mathcal{M}(X)\to\mathcal{M}(X)$  is defined by $T_*\mu=\mu \circ  T^{-1}$, which is continuous with respect to the weak-$*$ topology on $\mathcal{M}(X)$. The topological system $(\mathcal{M}(X), T _* ) $ is called the {\bf induced system} of $(X, T)$ on probability measures.

For each $n\in\mathbb{N}$, the {\bf $n$-th Bowen metrics} $d_n$ on $X$ and $D^*_n$ on $\mathcal{M}(X)$ are defined, respectively, by
\begin{align*}
	d_n(x,y) &:= \max_{0\leq j<n}d(T^jx,T^jy),\quad\forall x,y\in X;\\
	D_n^*(\mu,\nu) &:= \max_{0\leq j<n}D(T_*^j\mu,T_*^j\nu),\quad\forall \mu,\nu\in M(X).
\end{align*}
For every $r> 0$ we denote by $B_n (x, r)$, $\overline{B}_n(x,r)$ the open (resp. closed) ball of radius $r$ in the metric $d_n$ around $x$, i.e.
\[
B_n(x,r):=\{y\in X : d_n(x,y)<r\},\quad
\overline{B}_n(x,r):=\{y\in X : d_n(x,y)\le r\}.
\]
While $B_{D,n}^*(\mu,r)$ denotes the open ball of radius $r$ in the metric $D^*_n$ around $\mu$, i.e.
\[
B_{D,n}^*(\mu,r)=\{ \nu\in \mathcal{M}(X): D_n^*(\mu,\nu)<r\}.
\]
\subsection{Upper capacity  topological entropy}
In this subsection, we adopt the definition of upper capacity  topological entropy due to Feng and Huang~\cite[Section~2.1]{FengHuang}. 

Let $Z\subset X$ be a non-empty set. For $n\in\N$ and $r>0$, a set $E\subset Z$ is called a {\bf $(n,r)$-separated set} of $Z$, if $x,y\in E,\ x\ne y$ implies $d_n(x,y)>r$; a set $F\subset X$ is called a {\bf  $(n,r)$-spanning set} for $Z$,  if for any
$x\in Z$, there exists $y\in F$ satisfies $d_n(x,y)\le r$.  Denote by
$s_n(T,Z,r)$ the largest cardinality of a $(n,r)$-separated subset of
$Z$, and by $r_n(T,Z,r)$ the smallest cardinality of a $(n,r)$-spanning
set for $Z$.  Both numbers are finite since $(X,d_n)$ is compact.
The {\bf upper capacity topological entropy }of $Z$ is defined as
\[
\hUC(T,Z):=\lim_{r\to0}\limsup_{n\to\infty}
\frac1n\log r_n(T,Z,r)=\lim_{r\to0}\limsup_{n\to\infty}
\frac1n\log s_n(T,Z,r).
\]

Although we use separated and spanning sets to define the upper-capacity topological entropy, the resulting quantity agrees with its open-cover counterpart. This follows from the standard comparison between separated, spanning, and open-covering numbers \cite{AKM65,FengHuang,Walters}.

\subsection{Packing  topological entropy}
Packing topological entropy for arbitrary subsets and its fundamental properties were introduced by Feng and Huang~\cite[Section~2.3]{FengHuang}, and it is a dynamical analogue of packing dimension.

Let $Z\subset X$ be a non-empty set. 
For $s\geq0$, $N\in\N$, and $r>0$, let
\[
P_{N,r}^s(T,Z):=\sup\sum_i e^{-sn_i},
\]
where the supremum is taken over all finite or countable pairwise disjoint families $\{\overline B_{n_i}(x_i,r)\}_i$ with $x_i\in Z$ and
$n_i\geq N$ for all $i\in\N$. 

Since $P_{N,r}^s(T,Z)$ is non-increasing in $N$, hence the following limit exists:
\[
P_r^s(T,Z):=\lim_{N\to\infty}P_{N,r}^s(T,Z).
\]
Define
\[
\cP_r^s(T,Z):=\inf\left\{\sum_{k=1}^\infty P_r^s(T,Z_k):
Z\subset\bigcup_{k=1}^\infty Z_k\right\},
\]
where the infimum is taken over all countable covers
$\{Z_j\}_{j\ge1}$ of $Z$. There exists a critical value $\hP(T,Z,r)\in[0,+\infty]$ such that
\[
 \cP_r^s(T,Z)
=
\begin{cases}
	+\infty, & 0\le s<\hP(T,Z,r);\\
	0,       & s>\hP(T,Z,r).
\end{cases}
\]
The {\bf packing topological entropy of $Z$} is defined as
\[
\hP(T,Z)
:=
\lim_{r\to 0}\hP(T,Z,r).
\]
The limit exists in $[0,+\infty]$, since $\hP(T,Z,r)$ is non-decreasing  as $r\to 0$.

\subsection{Bowen topological entropy}

Bowen topological entropy for arbitrary subsets was introduced by
Bowen~\cite{Bowen1973}, and it is a dynamical analogue of Hausdorff
dimension. Its basic properties and variational principles were further
studied by Feng and Huang~\cite{FengHuang}.

Let $Z\subset X$ be a non-empty set. For $s\geq0$, $N\in\N$, and
$r>0$, define
\[
\Lambda_{N,r}^s(T,Z)
:=
\inf\left\{
\sum_i e^{-sn_i}:
Z\subset\bigcup_i B_{n_i}(x_i,r),\quad
x_i\in X,\quad n_i\geq N
\right\},
\]
where the infimum is taken over all finite or countable families of
Bowen balls $\{B_{n_i}(x_i,r)\}_i$ covering $Z$.

Since $\Lambda_{N,r}^s(T,Z)$ does not decrease as $N\to\infty$ and  $r\to 0$, the following two limits exists:
\[
\Lambda_r^s(T,Z)
:=
\lim_{N\to\infty}\Lambda_{N,r}^s(T,Z),\qquad \Lambda^s(T,Z)
:=
\lim_{r\to 0}\Lambda_{r}^s(T,Z),
\]
 The {\bf Bowen topological entropy of $Z$}  is defined as the critical value $\hB(T,Z)\in[0,+\infty]$ such that
\[
\Lambda^s(T,Z)
=
\begin{cases}
	+\infty, & 0\leq s<\hB(T,Z),\\
	0,       & s>\hB(T,Z).
\end{cases}
\]
The {\bf Bowen topological entropy} of $Z$ is defined as
\[
\hB(T,Z)
:=
\lim_{r\to0}\hB(T,Z).
\]
\subsection{Basic properties and variational principles for local etropies}For every non-empty subset  $Z\subseteq X$, its Bowen, packing, and upper-capacity topological entropies are independent of the choice of a compatible metric on the compact space  $X$.
As shown by Feng and Huang \cite[Proposition~2.1]{FengHuang}, these subset entropies satisfy several fundamental properties.
\begin{prop}
	\label{prop:entropy-properties}
	The following statements hold:
	
	\begin{enumerate}[label=(\roman*)]
		\item If $Z\subseteq Z'$, then
		\[
		\hUC(T,Z)\le \hUC(T,Z'),\
		\hB(T,Z)\le \hB(T,Z'),\
		\hP(T,Z)\le \hP(T,Z').
		\]
		
		\item For $Z\subseteq \bigcup_{i=1}^{\infty}Z_i$,  we have
		\[
		\hB(T,Z)\le \sup_{i\ge 1}\hB(T,Z_i).\quad
		\]
		
		\item For any $Z\subseteq X$,
		\[
		\hB(T,Z)\le \hP(T,Z)\le \hUC(T,Z).
		\]
		
		\item Furthermore, if $Z$ is $T$-invariant and compact, then
		\[
		\hB(T,Z)=\hP(T,Z)=\hUC(T,Z).
		\]
	\end{enumerate}
\end{prop}

For $\mu\in\M(X)$ define the {\bf measure-theoretical upper and lower entropies} of $\mu$ respectively by
\[
\overline h_\mu(T):=\int \overline h_\mu(T,x)d\mu(x),
\qquad
\underline h_\mu(T):=\int \underline h_\mu(T,x)d\mu(x),
\]
where
\[
 \overline h_\mu(T,x):=\lim_{r\to0}\overline h_\mu(T,x,r),
 \qquad
 \underline h_\mu(T,x):=\lim_{r\to0}\underline h_\mu(T,x,r)
\]
and
\[
 \overline h_\mu(T,x,r):=\limsup_{n\to\infty}
 -\frac1n\log\mu(B_n(x,r)),
\]
\[
 \underline h_\mu(T,x,r):=\liminf_{n\to\infty}
 -\frac1n\log\mu(B_n(x,r)).
\]

We use the following compact-set variational principles, proved by
Feng and Huang in~\cite[Theorems~1.2(i) and Theorem~1.3(i)]{FengHuang}.

\begin{theorem}
\label{thm:feng-huang}
Let $(X,T)$ be a  topological dynamical system, for every non-empty compact set $K\subset X$, then
\[
 \begin{aligned}
 \hB(T,K)&=\sup\{\underline h_\mu(T):\mu\in\M(K)\},\\
 \hP(T,K)&=\sup\{\overline h_\mu(T):\mu\in\M(K)\}.
 \end{aligned}
\]
\end{theorem}

\section{Proof of Theorm \ref{thm:intro-upper-capacity}}
\label{sec: proof of thm1.1}
Let $(X,d)$ be a compact metric space and $m\in\N_+$. The   product metric $$d^{\times m}((x_1,\ldots,x_m),(y_1,\ldots,y_m))=\max_{1\le i\le m}d(x_i,y_i)$$ on $X^m$ is compatible with the product topology. The induced product map $T^{\times m}: X^m \to X^m$
is defined by
\[
T^{\times m}(x_1,\ldots,x_m):=(Tx_1,\ldots,Tx_m).
\]
For $m\geq1$, we set
\[
 a_i:=\frac{2^{i-1}}{2^m-1},\qquad 1\leq i\leq m,
\]
and define
\begin{equation}
 \Phi_m(x_1,\ldots,x_m):=\sum_{i=1}^m a_i\delta_{x_i}.
 \label{eq:atomic-embedding}
\end{equation}

\begin{lemma}
\label{lem:atomic-embedding}
The map $\Phi_m\colon X^m\to\M(X)$ is a topological embedding and
\[
 T_*\circ\Phi_m=\Phi_m\circ T^{\times m}.
\]
Moreover, $\Phi_m(K^m)\subset\M(K)$ for every $K\subset X$.
\end{lemma}

\begin{proof}
Since distinct subsets of $\{1,2,4,\ldots,2^{m-1}\}$ have distinct sums, the map  $\Phi_m$ is injective.  For any
$f\in C(X)$, we have
$
 \int f\,\dd\Phi_m(\boldsymbol{x})=\sum_{i=1}^m a_i f(x_i),
$
thus $\Phi_m$ is continuous. The commutativity of the diagram follows immediately from
$T_*\delta_x=\delta_{Tx}$, and the inclusion $\Phi_m(K^m)\subset \mathcal{M}(K)$ is immediately comes
from the defintion of $\M (K)$.
\end{proof}

Theorem \ref{thm:intro-upper-capacity} follows immediately from Proposition \ref{prop:upper-amplification} and Theorem \ref{thm:upper-converse} below.
\begin{prop}
\label{prop:upper-amplification}
Let $(X,T)$ be a  topological dynamical system. For any non-empty compact $K\subset X$ and  $m\geq1$, we have
\[
 \hUC(T_*,\M(K))\geq m\,\hUC(T,K).
\]
Furthermore, $\hUC(T,K)>0$ implies that $\hUC(T_*,\M(K))=\infty$.
\end{prop}

\begin{proof}
Equip $X^m$ with the product metric $d^{\times m}$.  If $E$ is a
$(n,r)$-separated set of $K$, then $E^m$ is a $(n,r)$-separated set of $K^m$.  Hence
\begin{equation*}
 \hUC(T^{\times m},K^m)\geq m\,\hUC(T,K).
\end{equation*}

By Lemma~\ref{lem:atomic-embedding}, $\Phi_m$ conjugates
$(X^m,T^{\times m})$ to the compact invariant subsystem
$\Phi_m(X^m)$ of $(\M(X),T_*)$.  Proposition~\ref{prop:entropy-properties}
and monotonicity therefore give
\[
 \begin{split}
 \hUC(T_*,\M(K))
 &\geq\hUC(T_*,\Phi_m(K^m))\\
 &=\hUC(T^{\times m},K^m)
 \geq m\,\hUC(T,K).
 \end{split}
\]
\end{proof}

We will prove the reverse positivity implication via the  following combinatorial lemma of Glanser and Weiss \cite[Proposition~2.1]{GlasnerWeiss}.
\begin{lemma}\label{lemma GlasnerWeiss}
	For every $\varepsilon,b>0$ there are $N_0\in\N$ and $c_0>0$ such
	that the following holds for $N\geq N_0$.  If
	$\Phi\colon\ell_1^m\to\ell_\infty^N$ is a linear map with
	\[\|\Phi\|=\sup\{\|\Phi(x)\|_{\infty}: x\in\ell_1^m,\|x\|\leq 1\}\leq 1,\]
	and if
	$\Phi(B_1(\ell_1^m))$ contains more than $2^{bN}$ points that are
	$\varepsilon$-separated, then $m\geq2^{c_0N}$, where $B_1(\ell_1^m)=\{y\in\ell_1^m:\|y\|\leq1\}$.
\end{lemma}
\begin{theorem}
\label{thm:upper-converse}
Let $(X,T)$ be a  topological dynamical system. For every non-empty compact $K\subset X$,
\[
 \hUC(T_*,\M(K))>0\quad\Longrightarrow\quad\hUC(T,K)>0.
\]
\end{theorem}

\begin{proof}
Assume  that $\hUC(T_*,\M(K))>0$.  There are $\varepsilon,a>0$ and an infinite set $\mathcal{N}\subset\N$  such that for any
$n\in\mathcal{N}$, $\M(K)$ contains an $(n,\varepsilon)$-separated set $E_n$ with
\begin{equation*}
 |E_n|>e^{an}.
 \label{eq:converse-many-measures}
\end{equation*}

Choose $L\in\N_+$ large enough such that $\sum_{\ell>L}2^{-\ell}<\varepsilon/32$. Since $g_{\ell}$ is continuous for any $\ell\in\N_+$, there exists $\delta>0$ such that $d(x,y)<\delta$ yields that $|g_{\ell}(x)-g_{\ell}(y)|<\varepsilon/8$ for all $1\leq\ell\leq L$.

Let $M_n=r_n(T,K,\delta)$ and choose an $(n,\delta)$-spanning set
$\{z_1,\ldots,z_{M_n}\}\subset X$.  Its closed Bowen balls yield a
Borel partition $K_1,\ldots,K_{M_n}$ of $K$, with
$K_i\subset\overline B_n(z_i,\delta)$.  Define a linear map
$ \Phi_n\colon\ell_1^{M_n}\longrightarrow\ell_\infty^{nL}$
by
\begin{equation*}
 \Phi_n(\{v_i\}_{i=1}^{M_n})=
 \{2^{-\ell}\sum_{i=1}^{M_n}v_i g_\ell(T^jz_i)
 \}_{\substack{1\leq\ell\leq L, \ 0\leq j<n}}.
 \label{eq:converse-linear-map}
\end{equation*}
Since $\|g_\ell\|_\infty\leq1$, one has $\|\Phi_n\|\leq1$.

We now prove that for distinct $\nu,\mu\in E_n$, the following vectors in $\Phi_n(B_1(\ell_1^{M_n}))$ are $\frac{\varepsilon}{2^{L+4}}$-separated:
	\begin{equation*}
		\Phi_n(\nu(K_1),\ldots,\nu(K_{M_n}) \text{ and } \Phi_n((\mu(K_1),\ldots,\mu(K_{M_n})).
	\end{equation*}
	Otherwise, for any $1\leq \ell\leq L$ and $0\leq j<n$ we have
	\begin{equation*}
		\frac{\left|\sum_{i=1}^{M_n}\nu(K_i) g_{\ell}(T^{j}z_i)-\sum_{i=1}^{M_n}\mu(K_i) g_{\ell}(T^{j}z_i)\right|}{2^{\ell}}\leq \frac{\varepsilon}{2^{L+4}}.
	\end{equation*}
	Furthermore, for any $1\leq \ell\leq L$ and $0\leq j<n$, we have that
\begin{equation*}
		\begin{split}
			&\left|\int g_{\ell}(T^jx)d\nu(x)-\int g_{\ell}(T^jx)d\mu(x)\right|\\
			&\leq \left|\int g_{\ell}(T^jx)d\nu(x)-\sum_{i=1}^{M_n}\nu(K_i)g_{\ell}(T^jz_i)\right|+
			\left|\sum_{i=1}^{M_n}\nu(K_i)g_{\ell}(T^jz_i)-\sum_{i=1}^{M_n}\mu(K_i)g_{\ell}(T^jz_i)\right|\\
			&+\left|\int g_{\ell}(T^jx)d\mu(x)-\sum_{i=1}^{M_n}\mu(K_i)g_{\ell}(T^jz_i)\right|\\
			&=\left|\sum_{i=1}^{M_n}\int_{K_i} \left(g_{\ell}(T^jx)-g_{\ell}(T^jz_i)\right)d\nu(x)\right|+
			\left|\sum_{i=1}^{M_n}\nu(K_i)g_{\ell}(T^jz_i)-\sum_{i=1}^{M_n}\mu(K_i)g_{\ell}(T^jz_i)\right|\\
			&+\left|\sum_{i=1}^{M_n}\int_{K_i} \left(g_{\ell}(T^jx)-g_{\ell}(T^jz_i)\right)d\mu(x)\right|\\
			&\leq \frac{\varepsilon}{8}+\frac{\varepsilon}{16}+\frac{\varepsilon}{8}=\frac{5\varepsilon}{16}.
		\end{split}
	\end{equation*}
	This implies that for all $0\leq j<n$, we have 
	\begin{equation*}
		\begin{split}
			D(T_*^j\nu,T_*^j\mu)&=\sum_{\ell=1}^{\infty}2^{-\ell}\left|\int g_{\ell}(T^jx)d\nu(x)-\int g_{\ell}(T^jx)d\mu(x)\right|\\
			&\leq \sum_{\ell=1}^{L}2^{-\ell} \cdot \frac{5\varepsilon}{16}+ \sum_{\ell=L+1}^{\infty}2^{-\ell}\cdot 2\\
			&\leq \frac{5\varepsilon}{16}+\frac{\varepsilon}{16}=\frac{3\varepsilon}{8}<\frac{\varepsilon}{2}
		\end{split}
	\end{equation*}
Thus $D_n(T_*\nu,T_*\mu) <\frac{\varepsilon}{2}$.
 which leads to a contradiction with the fact $E_n$ is a $(n,\varepsilon)$-separated set. 
 
 Since $|E_n|\geq e^{an/2}$, by Lemma \ref{lemma GlasnerWeiss}, there exists $c_0>0$ and $N_0\in\N$ such that for every $n\in\mathcal{N}$ with $Ln\geq N_0$ we have $M_n\geq 2^{c_0nL}$. Since $\mathcal{N}$ is an infinite set,  we conclude that 
\[
 \hUC(T,K)\geq
 \limsup_{n\to\infty}\frac1n\log r_n(T,K,\delta)
 \geq c_0L\log2>0.
\]
\end{proof}

\begin{proof}[Proof of Theorem \ref{thm:intro-upper-capacity}]
The second equivalence in Theorem  \ref{thm:intro-upper-capacity} is obtained by combining Proposition \ref{prop:upper-amplification} and Theorem~\ref{thm:upper-converse}   above, and the zero statement follows by taking
contrapositives.
\end{proof}

\section{Proof of Theorem \ref{thm:intro-packing}}
\label{sec:proof of thm1.2}
\begin{lemma}
	\label{lem:large-core}
	Let $(X,T)$ be a  topological dynamical system. Let $\mu\in\M(X)$ with $\overline h_\mu(T)=0$.  For every
	$\eta>0$ and every $\kappa>0$, there exists a Borel set $A\subset X$
	such that
	\begin{equation*}
		\mu(A)>1-\kappa
		\qquad\text{and}\qquad
		\hUC(T,A)<\eta.
	\end{equation*}
\end{lemma}

\begin{proof}
	Since $\overline h_\mu(T)=0$, we have $\mu(\{x\in X: \overline h_\mu(T,x)=0\})=1$.
	Fix a sequence $\varepsilon_k\to 0$ and let $s=\eta/2$.  For
	$k,N\in\N_+$, define
	\begin{equation*}
		G_{k,N}:=\bigcap_{n\ge N}
		\left\{x\in X:\mu(B_n(x,\varepsilon_k))>e^{-sn}\right\}.
	\end{equation*}
	Each $G_{k,N}$ is Borel.  Fix
	$k\in\N_+$ and choose $x\in X$ with $\overline h_\mu(T,x)=0$. Then
	\[
	0\le \limsup_{n\to\infty}-\frac1n
	\log\mu(B_n(x,\varepsilon_k))
	\le \overline h_\mu(T,x)=0<s.
	\]
	Hence $\mu(B_n(x,\varepsilon_k))>e^{-sn}$ holds for all
	sufficiently large $n\in\N$. Thus $\{x\in X: \overline h_\mu(T,x)=0\}\subset \bigcup_{N=1}^{\infty}G_{k,N}$, and  we have $\mu(\bigcup_{N=1}^{\infty}G_{k,N})=1$.
	Choose $N_k\in\N_+$ so that $\mu(G_{k,N_k})>1-\kappa 2^{-k}$.

	Let $A=\bigcap_{k=1}^{\infty}G_{k,N_k}$.
	Then $A$ is Borel and
	\[
	\mu(A^c)
	\le\sum_{k=1}^{\infty}\mu(G_{k,N_k}^c)
	<\sum_{k=1}^{\infty}\kappa2^{-k}=\kappa.
	\]
	
	 For any $n\ge N_k$, let $E\subset A$ be a maximal
	$(n,2\varepsilon_k)$-separated set of $A$. Then  $B_n(x,\varepsilon_k)$ and $B_n(y,\varepsilon_k)$
	are pairwise disjoint for distinct $x,y\in E$.  Since
	$E\subset A\subset G_{k,N_k}$, we have
	\[
	1\ge\sum_{x\in E}\mu(B_n(x,\varepsilon_k))
	>|E|e^{-sn},\ \forall n\ge N_k.
	\]
	Thus $|E|<e^{sn}$, for any $n\ge N_k$.  Obviously $E$ is also an
	$(n,2\varepsilon_k)$-spanning set of $A$. Then $r_n(T,A,2\varepsilon_k)\le |E|<e^{sn}$ for all $k\in\N_+$ and $n\ge N_k$. This yields that $\hUC(T,A)\le s=\eta/2<\eta$.
	
\end{proof}

\begin{lemma}
	\label{thm:local-to-capacity}
	Let $(X,T)$ be a  topological dynamical system. For every $a>0$
	and $\varepsilon>0$, there exists $c>0$ 
	such that the following statement holds.  Let
	$A\subset X$ be Borel and $\tau\in\M(\M(X))$ satisfy
	$\tau(\M(A))=1$. Denote
	\begin{equation*}
		E:=\left\{\nu\in\M(X):
		\limsup_{n\to\infty}-\frac1n
		\log\tau(B_{D,n}^*(\nu,\varepsilon))>a\right\}.
	\end{equation*}
	If $\tau(E)>0$, then $\hUC(T,A)\geq c$.
\end{lemma}

\begin{proof}
	Fix $a,\varepsilon>0$ and $A\subset X$ be Borel.
	For each $n\in\N$, we define 
	\[H_n:=\left\{\nu\in\M(A): \tau(B_{D,n}^*(\nu,\varepsilon))<e^{-an}\right\}.\]
	Each $H_n$ is Borel.  Notice that $E\cap\M(A)\subset\limsup_{n\to\infty}H_n$.  Since
	$\tau(E)>0$ and $\tau(\M(A))=1$, it follows that $\tau(\limsup_{n\to\infty}H_n)>0$.
	By the  Borel--Cantelli Lemma, we have $\sum_{n=1}^{\infty}\tau(H_n)
	=+\infty$. Thus, the set $\mathcal{N}=\{n\in\N: \tau(H_n)\geq e^{-an/2}\}$ is infinite. 
	
	Fix $n\in\mathcal{N}$. Let $E_n$ be a maximal $(n,\varepsilon/2)$-separated set of $ H_n$. Then $H_n\subset\bigcup_{\nu\in E_n}B^*_{D,n}(\nu,\varepsilon/2)$. Moreover, $$1\le\tau(H_n)\leq \sum_{\nu\in E_n}\tau(B^*_{D,n}(\nu,\varepsilon))<e^{-an}|E_n|.$$ We conclude that $|E_n|\geq e^{an/2}$ for all $n\in\mathcal{N}$.
	
	Choose $L\in\N_+$ large enough such that $\sum_{\ell>L}2^{-\ell}<\varepsilon/32$. Since $g_{\ell}$ is continuous for any $\ell\in\N$, there exists $\delta>0$ such that $d(x,y)<\delta$ yields that $|g_{\ell}(x)-g_{\ell}(y)|<\varepsilon/8$ for all $1\leq\ell\leq L$.
	Let $M_n=r_n(T,A,\delta)$ and choose an $(n,\delta)$-spanning set
	$\{z_1,\ldots,z_{M_n}\}\subset X$.  Its closed Bowen balls yield a
	Borel partition $A_1,\ldots,A_{M_n}$ of $K$, with
	$A_i\subset\overline B_n(z_i,\delta)$.  
	
	We define a linear map
	$ \Phi_n\colon\ell_1^{M_n}\longrightarrow\ell_\infty^{nL}$
	by

	\[\Phi_n(\{v_i\}_{1\le i\le  M_n})=\{2^{-\ell}\sum_{i=1}^{M_n}x_i g_{\ell}(T^{j}z_i)\}_{1\leq\ell\leq L,0\leq j<n}.\]
	
	Since $\|g_\ell\|_\infty\leq1$, one has $\|\Phi_n\|\leq1$.
	We now prove that for distinct $\nu,\mu\in E_n$, the following vectors in $\Phi_n(B_1(\ell_1^{M_n}))$ are $\frac{\varepsilon}{2^{L+4}}$-separated:
	\begin{equation*}
		\Phi_n(\nu(A_1),\ldots,\nu(A_{M_n}) \text{ and } \Phi_n((\mu(A_1),\ldots,\mu(A_{M_n})).
	\end{equation*}
	Otherwise, for any $1\leq \ell\leq L$ and $0\leq j<n$ we have
	\begin{equation*}
		\frac{\left|\sum_{i=1}^{M_n}\nu(A_i) g_{\ell}(T^{j}z_i)-\sum_{i=1}^{M_n}\mu(A_i) g_{\ell}(T^{j}z_i)\right|}{2^{\ell}}\leq \frac{\varepsilon}{2^{L+4}}.
	\end{equation*}
	Furthermore, for any $1\leq \ell\leq L$ and $0\leq j<n$, we have that
	\begin{equation*}
		\begin{split}
			&\left|\int g_{\ell}(T^jx)d\nu(x)-\int g_{\ell}(T^jx)d\mu(x)\right|\\
			&\leq \left|\int g_{\ell}(T^jx)d\nu(x)-\sum_{i=1}^{M_n}\nu(A_i)g_{\ell}(T^jz_i)\right|+
			\left|\sum_{i=1}^{M_n}\nu(A_i)g_{\ell}(T^jz_i)-\sum_{i=1}^{M_n}\mu(A_i)g_{\ell}(T^jz_i)\right|\\
			&+\left|\int g_{\ell}(T^jx)d\mu(x)-\sum_{i=1}^{M_n}\mu(A_i)g_{\ell}(T^jz_i)\right|\\
			&=\left|\sum_{i=1}^{M_n}\int_{A_i} \left(g_{\ell}(T^jx)-g_{\ell}(T^jz_i)\right)d\nu(x)\right|+
			\left|\sum_{i=1}^{M_n}\nu(A_i)g_{\ell}(T^jz_i)-\sum_{i=1}^{M_n}\mu(A_i)g_{\ell}(T^jz_i)\right|\\
			&+\left|\sum_{i=1}^{M_n}\int_{A_i} \left(g_{\ell}(T^jx)-g_{\ell}(T^jz_i)\right)d\mu(x)\right|\\
			&\leq \frac{\varepsilon}{8}+\frac{\varepsilon}{16}+\frac{\varepsilon}{8}=\frac{5\varepsilon}{16}.
		\end{split}
	\end{equation*}
	This implies that for all $0\leq j<n$, we have 
	\begin{equation*}
		\begin{split}
			D(T_*^j\nu,T_*^j\mu)&=\sum_{\ell=1}^{\infty}2^{-\ell}\left|\int g_{\ell}(T^jx)d\nu(x)-\int g_{\ell}(T^jx)d\mu(x)\right|\\
			&\leq \sum_{\ell=1}^{L}2^{-\ell} \cdot \frac{5\varepsilon}{16}+ \sum_{\ell=L+1}^{\infty}2^{-\ell}\cdot 2\\
			&\leq \frac{5\varepsilon}{16}+\frac{\varepsilon}{16}=\frac{3\varepsilon}{8}<\frac{\varepsilon}{2}
		\end{split}
	\end{equation*}
	Thus $D_n(T_*\nu,T_*\mu) <\frac{\varepsilon}{2}$.
	which leads to a contradiction with the fact $E_n$ is a $(n,\varepsilon/2)$-separated set. 
	Since $|E_n|\geq e^{an/2}$, by Lemma \ref{lemma GlasnerWeiss}, there exists $c_0>0$ and $N_0\in\N$ such that for every $n\in\mathcal{N}$ with $nL\geq N_0$ we have $M_n\geq 2^{c_0Ln}$. Since $\mathcal{N}$ is an infinite set, we conclude that $$\hUC(T,A)\geq\limsup_{n\to\infty}\frac{\log M_n}{n}\geq c_0L\log 2>0.$$
\end{proof}

\begin{theorem}
	\label{thm:packing-converse}
	Let $(X,T)$ be a  topological dynamical system. Let $K\subset X$ be non-empty and compact.  If $\hP(T,K)=0$, then $\hP(T_*,\M(K))=0$.
\end{theorem}

\begin{proof}
	Assume by contradiction that $\hP(T,K)=0$ whereas $\hP(T_*,\M(K))>0$.
	Note $\M(K)$ is a non-empty compact
	subset of $\M(X)$.  We apply Theorem~\ref{thm:feng-huang} to $(\M(X),T_*)$ and $\M(K)$.  Then there exists $\tau\in\M(\M(X))$ satisfying $\tau(\M(K))=1$ and $\overline h_\tau(T_*)>0$.
	Moreover, there exist $a>0$,
	$\varepsilon>0$, and a Borel set $Q\subset\M(K)$ with
	$\tau(Q)>0$, such that for any $\nu\in Q$,
	\begin{equation}\label{eq:5.9}
		\limsup_{n\to\infty}-\frac1n
		\log\tau(B_{D,n}^*(\nu,\varepsilon))>a.
	\end{equation}
	Let $c=c(a,\varepsilon/2)>0$ be given by
	Lemma~\ref{thm:local-to-capacity}.
	
	We define $\bar\mu\in\M(X)$ by
	\begin{equation*}
		\bar\mu(B):=\int_{\M(X)}\nu(B)\,d\tau(\nu),
		\qquad B\in\mathcal{B}(X).
	\end{equation*}
	Since $\tau(\M(K))=1$, $\bar\mu(K)=\int_{\M(X)}\nu(K)\,d\tau(\nu)=1$.
	By Theorem \ref{thm:feng-huang}, we have $0\leq\overline h_{\bar\mu}(T)
	\leq \hP(T,K)=0$ and thus $\overline h_{\bar\mu}(T)=0$.
	Choose $0<\theta<\min\{1/2,\varepsilon/16\}$.
	By Lemma~\ref{lem:large-core}, we obtain a Borel set $A\subset X$ such that $\bar\mu(A^c)<\frac{\theta\tau(Q)}2$ and 
	\begin{equation}\label{eq:A-small-capacity}
		\hUC(T,A)<c.
	\end{equation}
	
	Using Markov's inequality,
	\begin{equation*}
		\tau(\{\nu\in\M(X):\nu(A^c)\geq\theta\}) \leq\frac1\theta\int_{\M(X)}\nu(A^c)\,d\tau(\nu)=\frac{\bar\mu(A^c)}\theta
		<\frac{\tau(Q)}2.
	\end{equation*}
	Consequently the Borel set $H:=Q\cap\{\nu\in\M(X):\nu(A^c)<\theta\}$
	satisfies
	\[
	\tau(H)\ge \tau(Q)-
	\tau(\{\nu\in\M(X):\nu(A^c)\ge\theta\})>\frac{\tau(Q)}2>0.
	\]  
	Let $\tau_H:=\frac{\tau|_H}{\tau(H)}$.
	Define $\mathcal R_A: H\to \M(A)$ by 
	\begin{equation*}
		\mathcal R_A(\nu):=\frac{\nu|_A}{\nu(A)}
	\end{equation*}
	for $\nu\in H$.
	This is well defined because
	$\nu(A)>1-\theta>1/2$ and $\mathcal R_A$ is Borel.
	
	Fix $\nu\in H$ and set $t=\nu(A^c)<\theta$.  Then $\|\nu-\mathcal R_A(\nu)\|_{\TV}=2t$, where $\|\cdot\|_{\TV}$ denotes the total variation of a signed measure.
	Moreover, for every $j\geq0$ and every $\ell\in\N$,
	\[
	\left|\int g_\ell\,d(T_*^j\nu)
	-\int g_\ell\,d(T_*^j\mathcal R_A(\nu))\right|\leq \|T_*^j(\nu-R_A(\nu))\|_{\TV}\leq \|\nu-R_A(\nu)\|_{\TV}=2t
	\]
	and thus $D(T_*^j\nu,T_*^j\mathcal R_A(\nu))\leq 2t<2\theta<\varepsilon/8$.
	This means $D_n^*(\nu,\mathcal R_A(\nu))<\varepsilon/8$ for any $n\in\N$ and $\nu\in H$.
	
	Let $\lambda=(\mathcal R_A)_*\tau_H$.
	Then $\lambda\in\M(\M(X))$ and $\lambda(\M(A))=1$.  Fix
	$\nu\in H$ and $n\in\N$.  If $\mu\in H$ satisfies $\mathcal R_A(\mu)
	\in B_{D,n}^*(\mathcal R_A(\nu),\varepsilon/2)$,
	then 
	\begin{align*}
		D_n^*(\mu,\nu)
		&\le D_n^*(\mu,\mathcal R_A(\mu))
		+D_n^*(\mathcal R_A(\mu),\mathcal R_A(\nu))
		+D_n^*(\mathcal R_A(\nu),\nu)\\
		&<\frac{\varepsilon}{8}+\frac{\varepsilon}{2}+\frac{\varepsilon}{8}<\varepsilon.
	\end{align*}
	Conclude $\mathcal R_A^{-1}
	\bigl(B_{D,n}^*(\mathcal R_A(\nu),\varepsilon/2)\bigr)
	\subset H\cap B_{D,n}^*(\nu,\varepsilon)$.
	
	By the definitions of $\lambda$ and $\tau_H$,
	\begin{equation*}
		\lambda(B_{D,n}^*(\mathcal R_A(\nu),\varepsilon/2))
		\leq\tau_H(B_{D,n}^*(\nu,\varepsilon))
		\leq\frac{\tau(B_{D,n}^*(\nu,\varepsilon))}{\tau(H)}
	\end{equation*}
	and by \eqref{eq:5.9} we further have
	\begin{equation}\label{eq:transferred-local-entropy}
		\begin{split}
			& \limsup_{n\to\infty}-\frac1n\log
			\lambda(B_{D,n}^*(\mathcal R_A(\nu),\varepsilon/2))\\
			&\geq \limsup_{n\to\infty}-\frac1n\left(\log
			\tau(B_{D,n}^*(\nu,\varepsilon))-\log\tau(H)\right)>a
		\end{split}
	\end{equation}
	for any $\nu\in H$.
	
	Note the set
	\[
	G:=\left\{\mu\in\M(A):
	\limsup_{n\to\infty}-\frac1n
	\log\lambda(B_{D,n}^*(\mu,\varepsilon/2))>a\right\}
	\]
	is Borel.   \eqref{eq:transferred-local-entropy} yields that
	$H\subset\mathcal R_A^{-1}(G)$.  Hence $\lambda(G)
	=\tau_H(\mathcal R_A^{-1}(G))=1$.
	By Theorem~\ref{thm:local-to-capacity}, we obtain that $\hUC(T,A)\geq c$,
	contradicting \eqref{eq:A-small-capacity}. 
\end{proof}


\begin{theorem}
\label{lem:dirac-lifting}
Let $(X,T)$ be a  topological dynamical system. Let $\iota(x)=\delta_x$ and, for $\mu\in\M(X)$, let
$\widehat\mu=\iota_*\mu\in\M(\M(X))$.  Then
\begin{equation}
 \overline h_{\widehat\mu}(T_*)=\overline h_\mu(T).
 \label{eq:dirac-local-equality}
\end{equation}
Consequently, for every non-empty compact $K\subset X$,
\begin{equation}
 \hP(T,K)\leq\hP(T_*,\M(K)).
 \label{eq:dirac-packing-lower-bound}
\end{equation}
\end{theorem}

 \begin{proof}
 	For every $f\in C(X)$, one has $\int f\,d\delta_x=f(x)$.  Hence
 	$\iota$ is continuous, and it is clearly injective.  Since $X$ is
 	compact and $\M(X)$ is Hausdorff, it is a homeomorphism from $X$ onto
 	the compact
 	set $\iota(X)$.  It is equivariant because $T_*\delta_x=\delta_{Tx}$.
 	Define a compatible metric on $X$ by
 	\[
 	\widetilde d(x,y):=D(\delta_x,\delta_y).
 	\]
 	Its $n$th Bowen metric is
 	\begin{equation*}
 		\widetilde d_n(x,y)=D_n^*(\delta_x,\delta_y).
 		\label{eq:dirac-bowen-isometry}
 	\end{equation*}
 	Fix $x\in X$, $r>0$, and $n\in\N$. Note that
 	\begin{align*}
 		\widehat\mu\bigl(B_{D,n}^*(\delta_x,r)\bigr)
 		=\mu\bigl(\iota^{-1}(B_{D,n}^*(\delta_x,r))\bigr)
 		=\mu\bigl(\{y:\widetilde d_n(x,y)<r\}\bigr).
 		\label{eq:dirac-ball-measure}
 	\end{align*}
 	Thus $\overline h^D_{\widehat\mu}(T_*,\delta_x)
 	=\overline h^{\widetilde d}_\mu(T,x)$, where the superscript means the metric used in the definition.
 	It remains to prove $\overline h^{\widetilde d}_\mu(T,x)=\overline h^{d}_\mu(T,x)$.
 	Since both $d$ and $\widetilde d$ generate the topology of the compact space $X$,
 	they are uniformly equivalent.  Given $r>0$, choose $s>0$ such that
 	$d(x,y)<s$ implies $\widetilde d(x,y)<r$. Then $B_{d,n}(x,s)\subset B_{\widetilde d,n}(x,r)$ for all $n\in\N$. 
 	Consequently,
 	\[
 	\limsup_{n\to\infty}-\frac1n
 	\log\mu(B_{\widetilde d,n}(x,r))
 	\le
 	\limsup_{n\to\infty}-\frac1n
 	\log\mu(B_{d,n}(x,s))
 	\le \overline h_\mu^{\,d}(T,x).
 	\]
 	Letting $r\to 0$, it follows that $\overline h_\mu^{\,\widetilde d}(T,x)\leq
 	\overline h_\mu^{\,d}(T,x)$.  Similarly, one can prove that $\overline h_\mu^{\,d}(T,x)\leq \overline h_\mu^{\,\widetilde d}(T,x)$. Conclude $\overline h_\mu^{\,\widetilde d}(T,x)=
 	\overline h_\mu^{\,d}(T,x)$ and thus $\overline h_{\widehat\mu}(T_*,\delta_x)
 	=\overline h_\mu(T,x)$.
 	Integrating with respect to $\widehat\mu=\iota_*\mu$ proves
 	\eqref{eq:dirac-local-equality}.
 	
 	Now let $K\subset X$ be non-empty and compact.  If $\mu(K)=1$, then
 	$\widehat\mu(\M(K))=1$, since $\delta_x\in\M(K)$ if and only if 
 	$x\in K$.  We apply Theorem~\ref{thm:feng-huang} to
 	$\M(K)$ and obtain that
 	\begin{align*}
 		\htop^P(T_*,\M(K))
 		&\ge \overline h_{\widehat\mu}(T_*)
 		=\overline h_\mu(T).
 	\end{align*}
 	Taking the supremum over all $\mu\in\M(X)$ with $\mu(K)=1$ and using
 	Theorem~\ref{thm:feng-huang} again proves
 	\eqref{eq:dirac-packing-lower-bound}.
 \end{proof}
\begin{proof}[Proof of Theorem \ref{thm:intro-packing}]
Theorem  \ref{thm:intro-packing} is obtained by combining Theorem \ref{thm:packing-converse} and Theorem~\ref{lem:dirac-lifting}   above.
\end{proof}

\section{Proof of Theorem \ref{thm:intro-bowen}}
\label{sec: proof of thm1.3} 
\begin{prop}
	\label{prop:bowen-product}
	For every non-empty compact $K\subset X$ and $m\geq1$,
	\[
	\hB(T^{\times m},K^m)\geq m\,\hB(T,K).
	\]
	Consequently, 
	\[
	\hB(T,K)>0 \implies \hB(T_*,\mathcal{M}(K))=\infty.
	\]
\end{prop}

\begin{proof}
	Choose $\mu\in\M(K)$ and denote by $\nu=\mu^{\otimes m}$ the product measure on $K^m$.  Obviously, $\nu\in\M(K^m)$.

    For every $\boldsymbol{x}=(x_1,\ldots,x_m)\in X^m$ and $r>0$, we have $B_n^{T^{\times m}}(\boldsymbol{x},r)
	=\prod_{i=1}^m B_n^T(x_i,r)$ and thus
\[
\begin{aligned}
\underline h_{\nu}
   (T^{\times m},\boldsymbol{x},r)
&=\liminf_{n\to\infty}
   -\frac1n
   \log \nu\bigl(B_n^{T^{\times m}}(\boldsymbol{x},r)\bigr)\\
&=\liminf_{n\to\infty}
   \sum_{i=1}^m
   \left(
   -\frac1n\log\mu\bigl(B_n^T(x_i,r)\bigr)
   \right)\\
&\geq \sum_{i=1}^m \liminf_{n\to\infty} \left(
   -\frac1n\log\mu\bigl(B_n^T(x_i,r)\bigr)
   \right)\\
&=\sum_{i=1}^m
   \underline h_\mu(T,x_i,r).
\end{aligned}
\]
Letting $r\to0$, we obtain that $h_{\nu}
   (T^{\times m},\boldsymbol{x})\geq \sum_{i=1}^m
   \underline h_\mu(T,x_i)$. 
Moreover, by Tonelli's theorem, we further have
\[
\begin{aligned}
\underline h_{\nu}(T^{\times m})
&=
\int_{X^m}
\underline h_{\nu}
   (T^{\times m},\boldsymbol{x})
\,d\nu(\boldsymbol{x})\\
&\geq
\int_{X^m}
\sum_{i=1}^m\underline h_\mu(T,x_i)
\,d\nu(x_1,\ldots,x_m)\\
&=
\sum_{i=1}^m
\left(
\int_X\underline h_\mu(T,x_i)\,d\mu(x_i)
\prod_{\substack{1\leq j\leq m\\ j\neq i}}
\int_X1\,d\mu(x_j)
\right)\\
&=
m\int_X\underline h_\mu(T,x)\,d\mu(x)\\
&=
m\,\underline h_\mu(T).
\end{aligned}
\]

We apply Theorem~\ref{thm:feng-huang} to $K^m$ and obtain that $\hB(T^{\times m},K^m)\geq \underline h_{\nu}(T^{\times m})\geq m \underline h_\mu(T)$. Using
Theorem~\ref{thm:feng-huang} again for $K$, we finally conclude that $\hB(T^{\times m},K^m)\geq m\,\hB(T,K)$ since $\mu\in\M(K)$ is arbitrary.

\end{proof}

\begin{lemma}\label{lemma:mass}
	Let $K\subset X$ be a Borel set and $\mu\in\M(K)$. If there exist $a\geq 0$ and $\varepsilon>0$ such that $\mu(B_n(x,\varepsilon))\leq e^{-an}$ for every $x\in K$ and $n\in\N$. Then $\hB(T,K)\geq a$.
\end{lemma}
\begin{proof}
	Fix $N\in\N$ and let $K\subset\bigcup_i B_{n_i}(z_i,\frac{\varepsilon}{2})$ be a cover of $K$ with $N\leq n_i\in\N$. Choose $y_i\in B_{n_i}(z_i,\frac{\varepsilon}{2})$ for each $i$. If $y\in K\cap B_{n_i}(z_i,\frac{\varepsilon}{2})$, then $d_{n_i}(y,y_i)<\varepsilon$. Conclude $K\cap B_{n_i}(z_i,\frac{\varepsilon}{2})\subset B_{n_i}(y_i,\varepsilon)$. Since $\mu(K)=1$,
	\begin{equation*}
		\begin{split}
			1\leq\sum_i\mu(K\cap B_{n_i}(z_i,\frac{\varepsilon}{2}))\leq\sum_i\mu(B_{n_i}(y_i,\varepsilon))\leq \sum_i e^{-an_i}
		\end{split}
	\end{equation*}
	and thus $\Lambda^a_{N,\varepsilon/2}(T,K)\geq 1$. Letting $N\to\infty$ and $\varepsilon\to0$, we have $\Lambda^a(T,K)\geq1$. Hence $\hB(T,K)\geq a$.
\end{proof}

\begin{theorem}
	\label{thm:example-bowen-reverse}
	There are a compact metric dynamical system $(X,T)$ and a non-empty
	compact $K\subset X$ such that
	\[
	\hB(T,K)=0,
	\qquad
	\hB(T_*,\M(K))=+\infty.
	\]
\end{theorem}

\begin{proof}
	Let $X=\{0,1\}^{\Nzero}$ and define $T: X\to X$ by $T((x_n)_{n\in\N})=(x_{n+1})_{n\in\N}$.
	Equip $X$ with the metric
	\begin{equation*}
		d(x,y)=\sum_{r=0}^{\infty}2^{-r-1}|x_r-y_r|.
	\end{equation*}
	Set $N_0=0$ and $N_m=2^{2^m}$ for $m\in\N$ and define
	\[
	J_0=\bigcup_{r\geq0}[N_{2r},N_{2r+1})\cap\Nzero,
	\qquad
	J_1=\bigcup_{r\geq0}[N_{2r+1},N_{2r+2})\cap\Nzero.
	\]
	Since $J_0\cap[0,N_{2r+2})\subset N_{2r+1}$ and $J_1\cap[0,N_{2r+1})\subset N_{2r}$ for any $r\in\N$ and $N_{m-1}/N_{m}=2^{-2^{m-1}}\to0$ as $m\to\infty$, it follows that
	\begin{equation}\label{density}
		\liminf_{n\to\infty}\frac{|J_0\cap [0,n)|}{n}=0,
		\qquad
		\liminf_{n\to\infty}\frac{|J_1\cap [0,n)|}{n}=0.
	\end{equation}

	For $i=0$ or $1$, set $K_i=\{x\in X:x_r=0\text{ for }r\notin J_i\}$ and $K=K_0\cup K_1$.
	Then $K$ is compact.  
	Fix
	$i\in\{0,1\}$, $\varepsilon>0$ and $n\in\N$.  Choose $L\in\N$ large enough such that
	$2^{-L-1}<\varepsilon$.
	If $x,y\in K_i$ satisfy $x_r=y_r$ for $r=1,\ldots,n+L$, then $d(T^jx,T^jy)\leq2^{-L-1}<\varepsilon$ for every $0\leq j<n$. That is to say, $d_n(x,y)<\varepsilon$.
	In this way, $K_i$ is covered by $(n,\varepsilon)$-Bowen balls with cardinality exactly $2^{|J_i\cap[0,n+L)|}$ for each $n\in\N$. 
	
	Fix $s>0$ and $N\in\N$. Then $\Lambda_{N,\varepsilon}^s(T,K_i)\leq 2^{|J_i\cap[0,n+L)|}e^{-ns}$ whenever $n\geq N$. By Proposition \ref{prop:entropy-properties}, we have $\Lambda_{\varepsilon}^s(T,K_i)=0$ and thus $\Lambda^s(T,K_i)=0$. Since $s>0$ is arbitrary, conclude $\hB(T,K_i)=0$ for $i=0$ and $1$. By Proposition \ref{prop:entropy-properties} we further have
	$\hB(T,K)=\max\{\hB(T,K_0),\hB(T,K_1)\}=0$.

	We now prove that $\hB(T_*,\M(K))=\infty$.
	Choose a dense sequence $(g_{\ell})_{\ell\in\N}$ in the unit ball of $C(X)$ with $g_1((x_n)_{n\in\N})=x_0$. Let $D$ be defined by \eqref{eq:weak* metric}. 
	For $p=(p_j)\in [0,1]^{\N}$ and $i=0,1$, denote by
	$\nu_p^{(i)}$ the product probability measure whose $j$-th coordinate
	has distribution $(1-p_j)\delta_0+p_j\delta_1$ if $j\in J_i$, and is
	$\delta_0$ otherwise.  Then $\nu_p^{(i)}(K_i)=1$.  Set
	\[
	\Phi(p):=\mu_p:=\frac12\nu_p^{(0)}+\frac12\nu_p^{(1)}.
	\]
	Obviously $\mu_p(K)=1$ and hence $\mu_p\in\M(K)$.  The map
	$\Phi$ is continuous. Moreover, for every $j\geq 0$
	\begin{equation*}
		\int_Xg_1 d(T_*^j\mu_p)=\int_X x_jd\mu_p=\frac{1}{2}p_j
	\end{equation*}
	and thus for any $p,q\in[0,1]^{\N}$ it holds
	\begin{equation}\label{eq:probability estimate}
		D(T_*^j\mu_p,T_*^j\mu_q)\geq \frac{1}{2}\left|\int_X g_1 d(T_*^j\mu_p)-\int_X g_1 d(T_*^j\mu_q)\right|=\frac{1}{4}|p_j-q_j|.
	\end{equation}
	If $\mu_p=\mu_q$, then $p_j=q_j$ for all $j\in\N$. Conclude the map $\Phi$ is injective.
	
	Fix $r\in\N$ and set 
	\begin{equation*}
		E_r:=\left\{0,\frac{1}{r},\ldots,\frac{r-1}{r},1\right\},
		\qquad
		F_r:=E_r^{\N}.
	\end{equation*}
	Let 
	\[\lambda_r:=\left(\frac{1}{r+1}\sum_{k=0}^{r}\delta_{\frac{k}{r}}\right)^{\otimes\N}\]
	be the product probability measure on $F_r$. Denote $C_r:=\Phi(F_r)$ and $\tau_r:=\Phi_*\lambda_r$.
	Obviously $\lambda_r(F_r)=1$ and $\tau_r(C_r)=1$. 
	
	Let $p,q\in F_r$ and $n\in\N$. By \eqref{eq:probability estimate}, if $D_n^*(\mu_p,\mu_q)<\frac{1}{8r}$ then $|p_j-q_j|<\frac{1}{2r}$ for all $0\leq j<n$. Note that either $p_j=q_j$ or $|p_j-q_j|\geq\frac{1}{r}$ for $p,q\in F_r$. Conclude $p_j=q_j$ for all $0\leq j<n$ whenever $D_n^*(\mu_p,\mu_q)<\frac{1}{8r}$. Thus for every $p\in F_r$ and $n\in\N$
	\begin{equation*}
		\begin{split}
			\tau_r(B_{D,n}^*(\mu_p,\frac{1}{8r}))&=\lambda_r(\Phi^{-1}(B_{D,n}^*(\mu_p,\frac{1}{8r})))\\
			&\leq\lambda_r(\{q\in F_r: q_j=p_j \text{ for } 0\leq j<n\})=\frac{1}{(r+1)^n}.
		\end{split}
	\end{equation*}
	By Theorem \ref{lemma:mass} we have $\hB(T_*,\M(K))\geq\log(r+1)$. Since $r\in\N$ is arbitrary, conclude $\hB(T_*,\M(K))=\infty$.
\end{proof}
\begin{proof}[Proof of Theorem \ref{thm:intro-bowen}]
Theorem  \ref{thm:intro-bowen} is obtained by combining Proposition \ref{prop:bowen-product} and Theorem~\ref{thm:example-bowen-reverse}   above.
\end{proof}

\section*{Acknowledgements}
Q. Huo was partially supported by the National Key Research and Development Program of China 2024YFA1013600, the China Postdoctoral Science Foundation 2025M773065 and  Fundamental Research Funds for the Central Universities WK0010250102. X. Wang was partially supported by the National Key Research and Development Program of China 2024YFA1013600, the  Postdoctoral Fellowship Program and China Postdoctoral Science Foundation BX2026007, the China Postdoctoral Science Foundation 2025M783147 and  Fundamental Research Funds for the Central Universities.   Both authors are grateful to Wen Huang, Kairan Liu and Leiye Xu useful discussions.

\end{document}